\documentclass{amsart}
\usepackage{amssymb,amsmath,ifthen,tikz,cite}
\usepackage{calc}
\usepackage[a4paper]{geometry}
\usepackage[colorlinks]{hyperref}
\usepackage[capitalize]{cleveref}

\advance\textheight by \topskip
\newtheorem{thm}{Theorem}[section]
\newtheorem{lem}[thm]{Lemma}
\newtheorem{conj}[thm]{Conjecture}
\numberwithin{equation}{section}

\newcommand{\tg}{\tilde{g}}
\newcommand{\rmand}{\quad\hbox{ and }\quad}

\allowdisplaybreaks

\title[Embedding vs extendability of non-bipartite graphs]{Surface embedding of non-bipartite $k$-extendable graphs}

\author{Hongliang Lu}
\address{School of Mathematics and Statistics, Xi'an Jiaotong University, 710049 Xi'an, P. R. China}
\email{luhongliang@mail.xjtu.edu.cn}

\author{David G.L. Wang}
\address{School of Mathematics and Statistics, Beijing Institute of Technology, 102488 Beijing, P. R. China}
\email{david.combin@gmail.com}


\begin{document}

\maketitle

\begin{abstract}
We find the minimum number $k=\mu'(\Sigma)$ for any surface $\Sigma$, such that every $\Sigma$-embeddable non-bipartite graph is not $k$-extendable. In particular, we construct the so-called bow-tie graphs $C_6\bowtie P_n$, and show that they are $3$-extendable. This confirms the existence of an infinite number of $3$-extendable non-bipartite graphs which can be embedded in the Klein bottle.
\end{abstract}

\section{Introduction}\label{sec:intro}
A matching~$M$ of a graph~$G$ is said to be \emph{extendable} if~$G$ has a
perfect matching containing~$M$.
Much attention to the theory of matching extension has been paid since it was introduced by Plummer~\cite{Plu80} in~1980.
We recommend Lov\'asz and Plummer's book~\cite{LP86B} for an excellent survey of the matching theory, 
and~\cite{Plu08BC,YL09B} for recent progress.
Interests in the matching extensions of graphs embedded on surfaces began with the charming result~\cite{Plu88PB} 
that no planar graph is $3$-extendable. 
We refer the reader to Gross and Tucker's book~\cite{GT87B} for basic notions on topological graph theory;
see also~\cite{BW09B}.

Plummer~\cite{Plu88} considered the problem of determining the minimum integer $k$ such that every
$\Sigma$-embeddable graph is not $k$-extendable. 
Based on some partial results of Plummer,
Dean~\cite{Dean92} found the complete answer to this problem.

\begin{thm}[Dean, Plummer]\label{thm:k}
Let $\Sigma$ be a surface of characteristic~$\chi$.
Let $\mu(\Sigma)$ to be the minimum integer $k$
such that every $\Sigma$-embeddable graph is not $k$-extendable. 
Then we have
\begin{equation}\label{ans:mu}
\mu(\Sigma)=\begin{cases}
3,&\text{if the surface~$\Sigma$ is homeomorphic to the sphere};\\[3pt]
2+\lfloor \sqrt{4-2\chi}\rfloor,&\text{otherwise}.
\end{cases}
\end{equation}
\end{thm}

Its proof made a heavy use of the Euler contribution technique,
which dates back to Lebesgue~\cite{Leb40},
developed by Ore~\cite{Ore67},
and flourished by Ore and Plummer~\cite{OP69}.

In a previous paper~\cite{LW13X}, 
we extended \cref{thm:k} by finding the minimum integer $k$
such that there is no $\Sigma$-embeddable $(n,k)$-graphs,
where an $(n,k)$-graph is a graph whose subgraph obtained by removing any~$n$ vertices is $k$-extendable.
This paper continues the study of this embeddable-extendable type of problems.
We dig a little deeper by concentrating on non-bipartite graphs.
Here is our main result.

\begin{thm}\label{thm:k'}
Let $\Sigma$ be a surface of characteristic~$\chi$.
Let $\mu'(\Sigma)$ to be the minimum integer $k$
such that every $\Sigma$-embeddable non-bipartite graph is not $k$-extendable.
Then we have
\begin{equation}\label{ans:mu'}
\mu'(\Sigma)=\begin{cases}
4,
&\text{if $\chi\in\{-1,\,0\}$};\\[3pt]
\lfloor(7+\sqrt{49-24\chi})/4\rfloor,&\text{otherwise}.
\end{cases}
\end{equation}
\end{thm}

Non-bipartite graphs differ from bipartite graphs in many aspects, even if we are concerned with only matching problems. 
For instance, K\"onig theorem states that 
the maximum size of a matching in a bipartite graph equals the minimum size of a node cover; 
see Rizzi~\cite{Riz00} for a short proof.
Taking a triangle as the graph under consideration, 
one may see immediately that non-bipartite graphs do not admit this beautiful property in general.

Another example is on the algorithmic complexity.
Lakhal and Litzler~\cite{LL98} discovered a polynomial-time algorithm
for the problem of finding the extendability of a bipartite graph. 
It is still unknown that whether the same extendability problem for non-bipartite graphs can be solved in polynomial time or not;
see Plummer~\cite{Plu94}.

The sharp distinction between the appearances of \cref{ans:mu,ans:mu'},
also supports the above difference between bipartite and non-bipartite graphs in the theory of matching extensions.

A big part of the proof of \cref{thm:k'} is to show the $3$-extendability 
of some so-called bow-tie graphs. We think the family of bow-tie graphs is interesting also on its own right.
We will confirm the infinity of the number of $3$-extendable graphs 
which can be embedded onto the Klein bottle, and which are non-bipartite.
An infinity number of such, but bipartite, graphs, 
were constructed recursively by Aldred, Kawarabayashi, and Plummer~\cite{AKP08}.

This paper is organized as follows. 
In the next section we list necessary notions and notations,
as well as necessary known results in the field of surface embedding and matching extension of graphs.
The stand-alone \cref{sec:bowtie} is devoted to the extendability 
of the two families of Cartesian product graphs of paths and cycles,
and of the bow-tie graphs denoted as $C_6\bowtie P_n$. 
We also pose a conjecture for the $3$-extendability of the general bow-tie graphs.
In \cref{sec:pf} we establish \cref{thm:k'} with the aid of these extendability results.

\section{Preliminaries}\label{sec:preliminary}\label{sec:pre}

This section contains an overview of necessary notion and notation.
Let $G=(V,E)$ be a simple graph.
We denote the number $|V(G)|$ of vertices by $|G|$ for short.
Denote by~$\delta(G)$ the minimum degree of~$G$,
and by~$\kappa(G)$ the connectivity.

\subsection{The surface embedding}

A surface is a connected compact Hausdorff space 
which is locally homeomorphic to an open disc in the plane. 
If a surface $\Sigma$ is obtained from the sphere by adding some number $g$ 
of handles (resp., some number $\tg$ of cross-caps), 
then $\Sigma$ is said to be orientable of genus~$g$ 
(resp., non-orientable of non-orientable genus~$\tg$). 
We shall follow the usual convention of denoting the surface of genus~$h$ 
(resp., non-orientable genus~$k$) by~$S_h$ (resp., $N_k$).

For a general surface~$\Sigma$, 
let~$g(\Sigma)$ be the genus of~$\Sigma$.
The Euler characteristic~$\chi(\Sigma)$ is defined by
\[
\chi(\Sigma)=\begin{cases}
2-2g(\Sigma),&\text{if $\Sigma$ is orientable},\\
2-g(\Sigma),&\text{if $\Sigma$ is non-orientable}.
\end{cases}
\]
A \emph{$2$-cell} (or {\em cellular}) embedding of a graph $G$ onto a surface is
a drawing of the graph $G$ on the surface such that the edges of $G$ crosses 
only at the vertices of $G$, and that every face is homeomorphic to an open disk.
In this paper, we wording ``embedding'' always means cellular embedding.
We say that a graph $G$ is {\em $\Sigma$-embeddable}
if there exists an embedding of the graph~$G$ on the surface~$\Sigma$.
The minimum value~$g$ such that $G$ is $S_g$-embeddable
is said to be the {\em genus} of~$G$, denoted $g(G)$.
Any embedding of~$G$ on~$S_{g(G)}$
is said to be a {\em minimal (orientable) embedding}.
Similarly,
the minimum value~$\tg$ such that $G$ is $N_{\tg}$-embeddable
is said to be the {\em non-orientable genus} of~$G$,
denoted $\tg(G)$.
Any embedding of~$G$ on~$N_{\tg(G)}$
is said to be a {\em minimal (non-orientable) embedding}.
Working on minimal embeddings, one should notice
the following two fundamental results,
which are due to Youngs~\cite{You63} and Parsons et al.~\cite{PPPV87} respectively.

\begin{thm}[Youngs]
Every minimal orientable embedding of a graph is 2-cell.
\end{thm}

\begin{thm}[Parson, Pica, Pisanski, Ventre]
Every graph has a minimal non-orientable embedding which is 2-cell.
\end{thm}

The formula of non-orientable genera of complete graphs was found by
Franklin~\cite{Fra34} in~1934 for $K_7$,
and by Ringel~\cite{Rin54} in~1954 for the other~$K_n$.
Early contributors include Heawood, Tietze, Kagno, Bose, Coxeter, Dirac, and so on;
see~\cite{Rin54}.
The more difficult problem of finding the genera of complete graphs
has been explored by Heffter, Ringel, Youngs, Gustin, Terry, Welch, Guy, Mayer, and so on.
A short history can be found in the famous work~\cite{RY68} of Ringel and Youngs in~1968,
who settled the last case.
These formulas are as follows.

\begin{thm}\label{thm:g:K}
Let $n\ge5$. We have
\vskip3pt
\begin{itemize}
\item[(i)]
$\tg(K_7)=3$ and $\tg(K_n)=\lceil (n-3)(n-4)/6\rceil$ when $n\ne7$;
\vskip3pt
\item[(ii)]
$g(K_n)=\lceil (n-3)(n-4)/12\rceil$.
\end{itemize}
\end{thm}

\subsection{The Euler contribution}
Let $G\to\Sigma$ be an embedding of a graph~$G$ on the surface $\Sigma$.
Euler's formula states that 
\[
|G|-e+f=\chi(G),
\] 
where $e$ is the number of edges of $G$,
and $f$ is the number of faces in the embedding.
Let $x_i$ denote the size of the $i$th face containing~$v$, i.e.,
the length of its boundary walk.
The {\em Euler contribution} of~$v$ is defined to be
\[
\Phi(v)=1-{d(v)\over 2}+\sum_{i}{1\over x_i},
\]
where the sum ranges over all faces containing~$v$.
One should keep in mind that
a face may contribute more than one angle to a vertex. This can be seen from
the embedding of~$K_5$ on the torus.
From Euler's formula, in any embedding of a connected graph~$G$,
we have 
\[
\sum_v\Phi(v)=\chi(\Sigma).
\] 
Thus there exists a vertex~$v$ such that
\begin{equation}\label{def:Phi}
\Phi(v)\ge{\chi(\Sigma)\over |G|}.
\end{equation}
Such a vertex is said to be a {\em control point} of the embedding.
Definition~\eqref{def:Phi} implies the following lemma immediately,
see also~\cite[Lemma 2.5]{Plu88} or~\cite[Lemma 2.5]{Dean92} for its proof.

\begin{lem}\label{lem:ctrl}
Let~$G$ be a connected graph of at least $3$ vertices.
Let $G\to\Sigma$ be an embedding.
Let $v$ be a control point which is contained in $x$ triangular faces.
Then we have
\begin{equation}\label{ineq:ctrl}
{d(v)\over 6}\le {d(v)\over 4}-{x\over 12}\le 1-{\chi(\Sigma)\over |G|}.
\end{equation}
\end{lem}

\subsection{The matching extension}

Let $G$ be a graph and $k\ge0$.
A {\em $k$-matching} of~$G$ is a collection of~$k$ pairwise disjoint edges.
\emph{Perfect matchings} are $|G|/2$-matchings. 
A \emph{near perfect matching} of the graph~$G$ is a perfect matching of the graph~$G-v$
for some vertex~$v$ of~$G$.
The most basic result for perfect matchings is Tutte's theorem~\cite{Tut47}.

\begin{thm}[Tutte]\label{thm:Tutte}
A graph $G$ has a perfect matching
if and only if for every vertex subset $S$,
the subgraph $G-S$ has at most $|S|$ connected components
with an odd number of vertices.
\end{thm}

A $k$-matching of the graph $G$ is said to be {\em perfect} if~$G$ has exactly $2k$ vertices.
The graph~$G$ is said to be {\em $k$-extendable} if
\begin{itemize}
\item
it has a perfect matching, and
\smallskip\item
for any $k$-matching~$M$,
the graph~$G$ has a perfect matching containing~$M$.
\end{itemize}

The following basic property on the connectivity of extendable graphs
can be found in~\cite{Plu80}.

\begin{thm}[Plummer]\label{thm:ext:basic}
Let $k\ge0$ and let~$G$ be a connected $k$-extendable graph. 
Then $G$ is $(k+1)$-connected, and thus $\delta(G)\ge k+1$.
\end{thm}

Liu and Yu~\cite{LY93} found the following result for the extendability of Cartesian product graphs.

\begin{thm}[Liu-Yu]\label{thm:LiuYu}
Let $G_1$ be a $k$-extendable graph and 
$G_2$ be a connected graph. 
Then the Cartesian product $G_1\times G_2$ is $(k+1)$-extendable.
\end{thm}

Gy\"ori and Plummer~\cite{GP92} gave the following nice generalization.

\begin{thm}[Gy\"ori-Plummer]\label{thm:GP92}
The Cartesian product of a $k$-extendable graph 
and an $l$-extendable graph is $(k+l+1)$-extendable.
\end{thm}

Plummer~\cite{Plu88PB} also gave the famous result that no planar graph is $3$-extendable.
The next deeper result is due to Lou and Yu~\cite[Theorem 7]{LY04};
see also~\cite[Chap.~6]{YL09B}.

\begin{thm}[Lou, Yu]\label{thm:LY:bip}
If~$G$ is a $k$-extendable graph of order at most $4k$,
then either~$G$ is bipartite or the connectivity $\kappa(G)$ of~$G$ is at least $2k$.
\end{thm}

We also need the following result, whose proof can be found in~\cite{Dean92,LW13X}.

\begin{lem}\label{lem:d:x}
Let $k\ge1$.
Let~$G$ be a connected $k$-extendable graph embedded on a surface~$\Sigma$.
Let~$v$ be a vertex of~$G$ which is contained in $x$ triangular faces in the embedding.
Then we have
\[
d(v)\ge\begin{cases}
k+1+\lceil x/2\rceil,&\text{if $x\le 2k-2$},\\
2k+1,&\text{if $x\ge 2k-1$}.
\end{cases}
\]
\end{lem}

\section{The matching extension of special product graphs}\label{sec:bowtie}

In this section, we shall establish the $3$-extendability for some special graphs,
which will be used in handling some sporadic cases in proving \cref{thm:k'}.

Denote by $P_m$ the path having $m$ vertices.
Denote by $P_m\times P_n$, as usual, 
the Cartesian product graph of the paths~$P_m$ and~$P_n$.
We label its vertices by~$v_{i,j}$ (or by $v_{ij}$ if there is no confusion), 
from the northwest to the southeast,
where $i\in[m]$ and $j\in[n]$.
We use the notation~$R_i$ 
to denote the vertex set $\{v_{i1},\,v_{i2},\,\ldots,\,v_{in}\}$ of the $i$-th row;
and use the notation~$T_j$ 
to denote the vertex set $\{v_{1j},\,v_{2j},\,\ldots,\,v_{mj}\}$ of the $j$-th column.
We say that any edge in a row is \emph{horizontal}, 
and that any edge in a column is \emph{vertical}.
For convenience, we consider the first subscript~$i$ of the notation~$v_{ij}$ 
as modulo~$m$, and consider the second subscript~$j$ as modulo~$n$, i.e.,
\[
v_{i+km,\,j+hn}=v_{ij}
\qquad\text{for all $k,h\in\mathbb{Z}$}.
\]
It follows that
$R_{i+m}=R_i$ for all $i$,
and that
$T_{j+n}=T_j$ for all $j$.
Denote by $C_n$ the cycle having $n$ vertices.
We use the same way to label the vertices of the graphs $P_m\times C_n$ and $C_m\times C_n$.

For any positive integers $m$ and $n$,
we define the {\em bow-tie graph} $C_m\bowtie P_n$
to be the graph obtained from the graph $C_m\times P_n$
by adding the edges~$v_{i1}v_{m+2-i,\,n}$ for all $i\in[m]$.
From \cref{fig:N2},
it is easy to see that the graph $C_m\bowtie P_n$ is $N_2$-embeddable.
\begin{figure}[htbp]
\centering
\begin{tikzpicture}
\coordinate (v11) at (1,-1);
\coordinate (v12) at (3,-1); 
\coordinate (v13) at (5,-1);
\coordinate (v14) at (7,-1);
\coordinate (v15) at (9,-1);
\coordinate (v16) at (11,-1);
\coordinate (v21) at (1,-2);
\coordinate (v22) at (3,-2);
\coordinate (v23) at (5,-2);
\coordinate (v24) at (7,-2);
\coordinate (v25) at (9,-2);
\coordinate (v26) at (11,-2);
\coordinate (v31) at (1,-3);
\coordinate (v32) at (3,-3);
\coordinate (v33) at (5,-3);
\coordinate (v34) at (7,-3);
\coordinate (v35) at (9,-3);
\coordinate (v36) at (11,-3);
\coordinate (v41) at (1,-4);
\coordinate (v42) at (3,-4);
\coordinate (v43) at (5,-4);
\coordinate (v44) at (7,-4);
\coordinate (v45) at (9,-4);
\coordinate (v46) at (11,-4);
\coordinate (v51) at (1,-5);
\coordinate (v52) at (3,-5); 
\coordinate (v53) at (5,-5);
\coordinate (v54) at (7,-5);
\coordinate (v55) at (9,-5);
\coordinate (v56) at (11,-5);
\coordinate (v61) at (1,-6);
\coordinate (v62) at (3,-6); 
\coordinate (v63) at (5,-6);
\coordinate (v64) at (7,-6);
\coordinate (v65) at (9,-6);
\coordinate (v66) at (11,-6);
\coordinate (v71) at (1,-7);
\coordinate (v72) at (3,-7); 
\coordinate (v73) at (5,-7);
\coordinate (v74) at (7,-7);
\coordinate (v75) at (9,-7);
\coordinate (v76) at (11,-7);
\draw
(v11)node[above left]{$v_{11}$}
(v12)node[above left]{$v_{12}$}
(v13)node[above left]{$v_{13}$}
(v14)node[above right]{$v_{14}$}
(v15)node[above right]{$v_{15}$}
(v16)node[above right]{$v_{11}$}
(v21)node[above left]{$v_{21}$}
(v22)node[above left]{$v_{22}$}
(v23)node[above left]{$v_{23}$}
(v24)node[above right]{$v_{24}$}
(v25)node[above right]{$v_{25}$}
(v26)node[above right]{$v_{61}$}
(v31)node[above left]{$v_{31}$}
(v32)node[above left]{$v_{32}$}
(v33)node[above left]{$v_{33}$}
(v34)node[above right]{$v_{34}$}
(v35)node[above right]{$v_{35}$}
(v36)node[above right]{$v_{51}$}
(v41)node[above left]{$v_{41}$}
(v42)node[above left]{$v_{42}$}
(v43)node[above left]{$v_{43}$}
(v44)node[above right]{$v_{44}$}
(v45)node[above right]{$v_{45}$}
(v46)node[above right]{$v_{41}$}
(v51)node[above left]{$v_{51}$}
(v52)node[above left]{$v_{52}$}
(v53)node[above left]{$v_{53}$}
(v54)node[above right]{$v_{54}$}
(v55)node[above right]{$v_{55}$}
(v56)node[above right]{$v_{31}$}
(v61)node[above left]{$v_{61}$}
(v62)node[above left]{$v_{62}$}
(v63)node[above left]{$v_{63}$}
(v64)node[above right]{$v_{64}$}
(v65)node[above right]{$v_{65}$}
(v66)node[above right]{$v_{21}$}
(v71)node[above left]{$v_{11}$}
(v72)node[above left]{$v_{12}$}
(v73)node[above left]{$v_{13}$}
(v74)node[above right]{$v_{14}$}
(v75)node[above right]{$v_{15}$}
(v76)node[above right]{$v_{11}$};
\draw[thick]
(v11)--(v16)
(v21)--(v26)
(v31)--(v36)
(v41)--(v46)
(v51)--(v56)
(v61)--(v66)
(v71)--(v76)
(v11)--(v71)
(v12)--(v72)
(v13)--(v73)
(v14)--(v74)
(v15)--(v75)
(v16)--(v76);
\end{tikzpicture}
\caption{The bowtie graph $C_6\bowtie P_5$ is $N_2$-extendable.}\label{fig:N2}
\end{figure}

In the subsequent three subsections, 
we will explore the matching extension of the following Cartesian product graphs respectively:
\[
P_m\times C_n,\qquad
C_m\times C_n,\rmand
C_m\bowtie P_n.
\]
Precisely speaking, we will show that the graph $P_m\times C_n$ is $2$-extendable,
and the other two graphs are $3$-extendable, subject to some natural conditions on the integers~$m$ and~$n$.

Here we describe a combinatorial idea, which will be adopted in all the proofs uniformly.
Let~$G$ be a graph with a matching~$M$.
We say that~$G$ is \emph{separable} by a subgraph~$G'$ (with respect to~$M$), if 
\begin{itemize}
\smallskip\item
the matching~$M$ has at least one edge in the subgraph~$G'$; and
\smallskip\item
no edge of the matching~$M$ has ends in both of the subgraphs~$G'$ and $G-V(G')$.
\end{itemize}
\vskip 2pt
We call the subgraph $G'$ an~\emph{$M$-separator} of~$G$, if 
\begin{itemize}
\smallskip\item
the subgraph~$G'$ has a perfect matching containing the edge set~$M\cap E(G')$; and
\smallskip\item
the subgraph~$G-V(G')$ has a perfect matching containing the edge set~$M\cap E(G-V(G'))$.
\end{itemize}
\vskip 2pt
In particular, the subgraph $G-V(G')$ has a perfect matching even if the set $M\cap E(G-V(G'))$ is empty.
From the above definition, 
it is direct to see that the extendability of a matching~$M$ can be confirmed by finding an $M$-separator.
We call this approach the \emph{separator method}.
We will use it uniformly by choosing the separator to be a subgraph induced by consecutive rows or columns.

\subsection{The $2$-extendability of the graph $P_m\times C_n$}

This subsection is devoted to establish the following result. 
It is basic and will be used in the proof of \cref{lem:CC}, \cref{thm:CC}, and \cref{thm:k'}.

\begin{thm}\label{thm:PC}
Let $m,n\ge4$.
The Cartesian product graph $P_m\times C_n$ is $2$-extendable if and only if 
the integer~$m$ or~$n$ is even.
\end{thm}

\begin{proof}
The necessity is clear from the definition.
When~$n$ is even, the cycle $C_n$ is $1$-extendable. 
Thus the sufficiency is true from \cref{thm:LiuYu}.
It suffices to show the sufficiency for odd~$n$.

Let $n\ge 5$ be an odd integer. Then the integer~$m$ is even.
Let $G$ be the graph $P_m\times C_n$, with a $2$-matching~$M=\{e_1,e_2\}$.
Note that every column of the graph~$G$ is isomorphic to the path~$P_m$, which has a perfect matching.
We will adopt the separator method by finding some columns, 
whose induced subgraph has a perfect matching containing the matching~$M$.
We have $3$ cases to treat.

\medskip
\noindent{\it Case 1.}
There are two disjoint pairs of adjacent columns, 
such that one pair contains the vertex set~$V(e_1)$, and the other pair contains the vertex set~$V(e_2)$.
Since the subgraph induced by any two adjacent columns is $1$-extendable, the four columns form an $M$-separator.

\medskip
\noindent{\it Case 2.}
The vertex set~$V(M)$ is contained in two adjacent columns, and Case 1 does not occur.
Then the two adjacent columns form an $M$-separator.
In fact, when both the edges~$e_1$ and~$e_2$ are vertical and in distinct columns, 
the previous possibility happens, a contradiction.
In other words, either the vertex set~$V(M)$ is contained in one column, 
or one of the edges in the matching~$M$ is horizontal.

\medskip
\noindent{\it Case 3.}
Otherwise, the vertex set $V(M)$ intersects with exactly three consecutive columns,
and both the edges $e_1$ and $e_2$ are horizontal.

We proceed by induction on~$m$. For $m=4$, we have $2$ subcases to treat.

\medskip
\noindent{\it Case 1.}
Assume that the edges in the matching~$M$ lie in Row~1 and Row~2,
or in Row~1 and Row~3. Since every row is isomorphic to a cycle, 
we can suppose without loss of generality that 
\[
e_1=v_{11}v_{12}
\rmand
e_2\in\{v_{22}v_{23},\,v_{32}v_{33}\}.
\]
In this case, the subgraph $G[T_1,T_2,T_3,T_4]$ has the perfect matching
\[
\{v_{11}v_{12},\,
v_{13}v_{14},\,
v_{22}v_{23},\,
v_{32}v_{33},\,
v_{41}v_{42},\,
v_{43}v_{44},\,
v_{21}v_{31},\,
v_{24}v_{34}\},
\]
which contains the matching~$M$; see \cref{fig:P4Cn:1}.
\begin{figure}[htbp]
\centering
\begin{tikzpicture}
\draw[ultra thick]
(1,-1)--(2,-1)
(2,-2)--(3,-2)
(2,-3)--(3,-3);
\draw[ultra thin]
(3,-1)--(4,-1)
(1,-4)--(2,-4)
(3,-4)--(4,-4)
(1,-2)--(1,-3)
(4,-2)--(4,-3);
\draw
(1,-1)node[above=12pt]{$1$}
(2,-1)node[above=12pt]{$2$}
(3,-1)node[above=12pt]{$3$}
(4,-1)node[above=12pt]{$4$}
(1,-1)node[left=16pt]{$1$}
(1,-2)node[left=16pt]{$2$}
(1,-3)node[left=16pt]{$3$}
(1,-4)node[left=16pt]{$4$};
\end{tikzpicture}
\caption{The extension of the matching $M$ for Case 1.}\label{fig:P4Cn:1}
\end{figure}
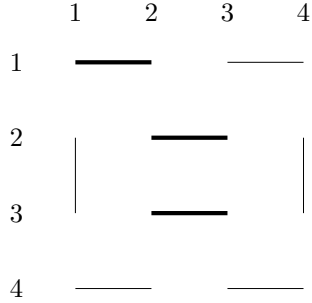

\medskip
\noindent{\it Case 2.}
Otherwise, the edges in the matching $M$ lie in Row~1 and Row~4,
or in Row~2 and Row~3. We can suppose that 
\[
M=\{v_{11}v_{12},\,v_{42}v_{43}\}
\qquad\text{or}\qquad
M=\{v_{21}v_{22},\,v_{32}v_{33}\}.
\]
In this case, the subgraph $G[T_1,T_2,T_3]$ has the perfect matching 
\[
\{v_{11}v_{12},\,
v_{21}v_{v22},\,
v_{13}v_{23},\,
v_{31}v_{41},\,
v_{32}v_{33},\,
v_{42}v_{43}\},
\]
which contains the matching~$M$; see \cref{fig:P4Cn:2}. 
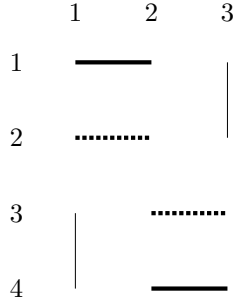
\begin{figure}[htbp]
\centering
\begin{tikzpicture}
\draw[ultra thick]
(1,-1)--(2,-1)
(2,-4)--(3,-4);
\draw[ultra thick, densely dotted]
(1,-2)--(2,-2)
(2,-3)--(3,-3);
\draw[very thin]
(3,-1)--(3,-2)
(1,-3)--(1,-4);
\draw
(1,-1)node[above=12pt]{$1$}
(2,-1)node[above=12pt]{$2$}
(3,-1)node[above=12pt]{$3$}
(1,-1)node[left=16pt]{$1$}
(1,-2)node[left=16pt]{$2$}
(1,-3)node[left=16pt]{$3$}
(1,-4)node[left=16pt]{$4$};
\end{tikzpicture}
\caption{The extension of the matching $M$ for Case 2.}\label{fig:P4Cn:2}
\end{figure}
This completes the proof for $m=4$.

Now we can suppose that $m\ge 6$,
and that the graph $P_{m-2}\times C_n$ is $2$-extendable.
Note that the subgraph induced by any two adjacent rows has a perfect matching. 
By induction, we are done if the matching~$M$ shares no vertices with the first two rows.
For the same reason, we are done if the matching~$M$ shares no vertices with the last two rows.
Since $m\ge 6$, we can suppose that the horizontal edge~$e_1$ is in the first two rows,
and that the horizontal edge~$e_2$ is in the last two rows.
On one hand, the subgraph~$G[R_1,R_2]$ is isomorphic to the graph $P_2\times C_n$, which is $1$-extendable.
On the other hand, the subgraph $G-R_1-R_2$ is isomorphic to the graph $P_{m-2}\times C_n$,
which is $2$-extendable by induction hypothesis.
Hence, the subgraph~$G[R_1,R_2]$ is an $M$-separator.
This completes the proof.
\end{proof}

\subsection{The $3$-extendability of the graph $C_m\times C_n$}

In this subsection we study the extendability of the graph $C_m\times C_n$,
which will be used in the proof of \cref{thm:k'}.

A necessary condition for the graph $C_m\times C_n$ to have a perfect matching is that
one of the integers~$m$ and~$n$ is even.
By symmetry, we can suppose that the integer~$m$ is even.
In virtue of \cref{thm:GP92}, the graph~$C_m\times C_n$ is $3$-extendable if the integer~$n$ is also even.
Therefore, we can suppose that~$n$ is odd. The following lemma will be used for several times.

\begin{lem}\label{lem:CC}
Let $m\ge6$ be an even integer, and let $n\ge5$ be an odd integer.
Let $G$ be the graph $C_m\times C_n$, with a $3$-matching~$M$.
Then the matching~$M$ is extendable if~$G$ is separable by

\smallskip\noindent(i) 
a subgraph $G[R_i,\,R_{i+1}]$ for some $i\in[m]$, 
which contains exactly one edge of the matching~$M$; or

\smallskip\noindent(ii) 
a subgraph $G[T_j,\,T_{j+1}]$ for some $j\in[n]$, 
which contains one or two edges of the matching~$M$.
\end{lem}

\begin{proof}
We prove the validities of Condition (i) and Condition (ii) individually.
Let $i\in[m]$ and $j\in[n]$.

\smallskip\noindent(i)
The subgraph $G[R_i,\,R_{i+1}]$ is isomorphic to the graph $P_2\times C_n$,
which is $1$-extendable.
The subgraph $G-R_i-R_{i+1}$ is isomorphic to the graph $P_{m-2}\times C_n$.
Since $m-2\ge 4$, then the subgraph $G-R_i-R_{i+1}$ is $2$-extendable by \cref{thm:PC}.
Hence, the matching~$M$ is extendable in the graph~$G$.

\smallskip\noindent(ii)
The subgraph $G[T_j,\,T_{j+1}]$ is isomorphic to the graph $C_m\times P_2$, 
and the subgraph $G-T_j-T_{j+1}$ is isomorphic to the graph $C_m\times P_{n-2}$.
Both of them are $2$-extendable by \cref{thm:GP92}.
Hence, the matching $M$ is extendable in the graph~$G$.
\end{proof}

Here the main result of this subsection.

\begin{thm}\label{thm:CC}
Let $m\ge6$ be an even integer, and let $n\ge 5$ be an odd integer.
Then the Cartesian product graph $C_m\times C_n$ is $3$-extendable.
\end{thm}

\begin{proof}
Let $m\ge6$ be an even integer, and let $n\ge 5$ be an odd integer.
Let $G$ be the graph $C_m\times C_n$, with a $3$-matching $M=\{e_1,e_2,e_3\}$.
In order to show that the matching~$M$ is extendable, it suffices to find 
\begin{itemize}
\smallskip
\item
a row index $i^*$ such that the subgraph $G[R_{i^*},\,R_{i^*+1}]$ is an $M$-separator, or
\smallskip
\item
a column index $j^*$ such that the subgraph $G[T_{j^*},\,T_{j^*+1}]$ is an $M$-separator.
\end{itemize}
Let~$h$ be the number of horizontal edges in the matching~$M$.
Then we have $h\in\{0,1,2,3\}$. We treat these $4$ cases individually.

\medskip
\noindent{\it Case 1.} $h=0$, that is, all edges in the matching~$M$ are vertical.
Note that each column of the graph~$G$ is isomorphic to the cycle~$C_m$, which is $1$-extendable.
If the $3$ edges in the matching~$M$ are in distinct columns, 
we are done immediately.
If they are in the same column, then that column together with one of its adjacent columns form an $M$-separator.
Otherwise, we can suppose that Column~$j$ contains the edges~$e_1$ and~$e_2$, but not the edge~$e_3$.
By virtue of \cref{lem:CC},
we can take $j^*=j$ if the edge~$e_3$ is not in Column $(j+1)$;
and take $j^*=j+1$ otherwise.

\medskip
\noindent{\it Case 2.} $h=1$.
In this case, we can suppose that the edge $e_1$ is horizontal,
and that the edges~$e_2$ and $e_3$ are vertical. 
Since each row is isomorphic to a cycle, 
we can further suppose that the edge~$e_1=v_{i1}v_{i2}$ intersects with the first two columns. 

If at most one of the edges~$e_2$ and~$e_3$ is in the first two columns,
then we can take $j^*=1$ by \cref{lem:CC}.
Otherwise, both of them are in the first two columns.
If the vertex set~$V(M)$ misses Row~$(i+1)$,
then we can take $i^*=i$ by \cref{lem:CC}.
For the same reason, we can take $i^*=i-1$ if the vertex set~$V(M)$ misses Row~$(i-1)$.
Otherwise, one of the vertical edges~$e_2$ and~$e_3$ is immediately above the horizontal edge~$e_1$,
and the other is immediately below the edge~$e_1$. In this case, we can take $i^*=i+1$ by \cref{lem:CC}.

\smallskip
\noindent{\it Case 3.} $h=2$.
We can suppose that the edges $e_1$ and $e_2$ are horizontal,
and that the edge $e_3$ is vertical. 
Furthermore, we can suppose that the horizontal edge~$e_1$ intersects with the first two columns, 
the horizontal edge~$e_2$ intersects with Column~$j$ and Column $(j+1)$,
and that the vertical edge~$e_3=v_{pq}v_{p+1,\,q}$, where $p\in[m]$ and $q\in[n]$.

If $j=1$, to wit, the horizontal edge~$e_2$ lies below the edge~$e_1$. In this case, we can take $j^*=1$.
In fact, since the edge~$e_3$ is vertical, the graph $G$ is separable by the subgraph $G[T_1,T_2]$. 
On the other hand, it is easy to see that both of the subgraphs $G[T_1,T_2]-V(e_1\cup e_2)$
and $G-T_1-T_2$ are $1$-extendable.

If $j\ge 3$, then the graph $G$ is separable by the first two columns with respect to the matching~$M$.
In this case, we can also take $j^*=1$, by using \cref{lem:CC}.

Otherwise, we have $j=2$, that is, 
the vertex set of the horizontal edges~$e_1$ and~$e_2$ intersects with exactly the first three columns.
\begin{itemize}
\item
If $q\in[3]$, i.e., the edge~$e_3$ is also contained in the subgraph $G[T_1,T_2,T_3]$,
then we can take $i^*=p$. In fact, the subgraph $G'=G[R_p,\,R_{p+1}]$ contains the vertical edge~$e_3$,
and possibly one of the horizontal edges~$e_1$ and~$e_2$.
In any case, the matching $M\cap E(G')$ is extendable in the subgraph~$G'$.
\item
If $q\ge 4$, i.e., the vertical edge~$e_3$ has empty intersection with the first three columns.
In this case, the subgraph~$G[T_q]$ is an $M$-separator.
In fact, the subgraph~$G[T_q]$, which contains the edge~$e_3$, is isomorphic to the cycle~$C_m$,
which is $1$-extendable. On the other hand, 
the subgraph $G-T_q$ is isomorphic to the graph $C_m\times P_{n-1}$, i.e.,
the graph $P_{n-1}\times C_m$. Since $n-1\ge4$,
it is $2$-extendable by \cref{thm:PC}.
\end{itemize}

\smallskip
\noindent{\it Case 4.} $h=3$.
If all edges in $M$ lie in the same row,
by \cref{lem:CC}, we can take $j^*=j$ for any edge $v_{ij}v_{i,j+1}\in M$.
Otherwise, there exists a row $R_i$
containing exactly one edge in $M$ such that one of its adjacent rows
has no edges in $M$. In other words, 
either $R_{i-1}\cap V(M)=\emptyset$ or $R_{i+1}\cap V(M)=\emptyset$.
By \cref{lem:CC}, we can take $i^*=i-1$ in the former case, and $i^*=i$ in the latter case.

This completes the proof.
\end{proof}

We remark that the graph $C_4\times C_n$ is not $3$-extendable when $n$ is odd.
This can be seen from the fact that the particular $3$-matching 
\[
M=\{v_{11}v_{12},\,v_{22}v_{32},\,v_{31}v_{41}\}
\]
is not extendable; see \cref{fig:C4Cn}. Define
\[
U=\{v_{i,2j}\,\colon\ i\in\{1,3\},\ 2\le j\le (n-1)/2\}
\cup\{v_{i,\,2j+1}\,\colon\ i\in\{2,4\},\ 1\le j\le (n-1)/2\}.
\]
We have $|U|=2n-4$. Note that the subgraph $G-V(M)-U$ consists of $2n-2$ isolated vertices.
By Tutte's theorem (see \cref{thm:Tutte}), the subgraph $G-V(M)$ has no perfect matchings.
\begin{figure}[htbp]
\centering
\begin{tikzpicture}
\draw[very thin, densely dotted] 
(1,-1) grid (8,-4)
(10,-1) grid (13,-4);
\draw 
(1,-1)node[left=6pt]{$1$}
(1,-2)node[left=6pt]{$2$}
(1,-3)node[left=6pt]{$3$}
(1,-4)node[left=6pt]{$4$}
(1,-1)node[above=6pt]{$1$}
(2,-1)node[above=6pt]{$2$}
(3,-1)node[above=6pt]{$3$}
(4,-1)node[above=6pt]{$4$}
(5,-1)node[above=6pt]{$5$}
(6,-1)node[above=6pt]{$6$}
(7,-1)node[above=6pt]{$7$}
(8,-1)node[above=6pt]{$8$}
(10,-1)node[above=6pt]{$n-3$}
(11,-1)node[above=6pt]{$n-2$}
(12,-1)node[above=6pt]{$n-1$}
(13,-1)node[above=6pt]{$n$}
(9,-2.5)node[]{$\ldots$};
\draw[ultra thick]
(1,-1)--(2,-1)
(2,-2)--(2,-3)
(1,-3)--(1,-4);
\foreach \x in {4,6,8,10,12} \foreach \y in {-1,-3} \fill[black,opacity=1] (\x,\y) circle (2pt);
\foreach \x in {3,5,7,11,13} \foreach \y in {-2,-4} \fill[black,opacity=1] (\x,\y) circle (2pt);
\end{tikzpicture}
\caption{The graph $C_4\times C_n$ is not $3$-extendable when $n$ is odd.}\label{fig:C4Cn}
\end{figure}
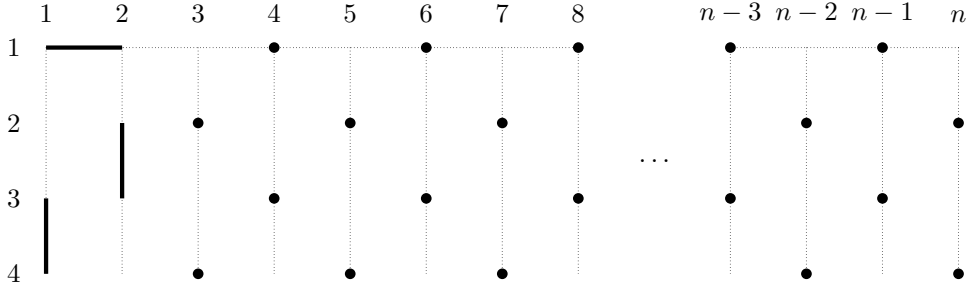

\subsection{The $3$-extendability of the graph $C_6\bowtie P_n$}

When the integer $n$ is even, the spanning subgraph $C_6\times P_n$ of the bow-tie graph $C_6\bowtie P_n$ 
is $3$-extendable by \cref{thm:GP92}. 
In this subsection, we will show the $3$-extendability of the graph $C_6\bowtie P_n$ for odd integers $n\ge5$.

It is easy to see that the graph $C_6\bowtie P_n$ can be drawn as in \cref{fig:bowtie}, which is symmetric up and down. 
For convenience, we rename the vertices in the following way:
\begin{align*}
&v_{1i}=h_i,\qquad v_{6i}=q_i,\qquad v_{2i}=q_{i+n},\\
&v_{4i}=h_i',\qquad v_{5i}=q_i',\qquad v_{3i}=q_{i+n}',
\end{align*}
and use capital letters to denote vertex subsets as follows:
\[
\begin{aligned}
H&=\{\,h_i\,\colon\,i\in[n]\,\},&\qquad 
Q&=\{\,q_j\,\colon\,j\in[2n]\,\},&\qquad 
J&=H\cup Q,\\
H'&=\{\,h_i'\,\colon\,i\in[n]\,\},&\qquad 
Q'&=\{\,q_j'\,\colon\,j\in[2n]\,\},&\qquad 
J'&=H'\cup Q'.
\end{aligned}
\]

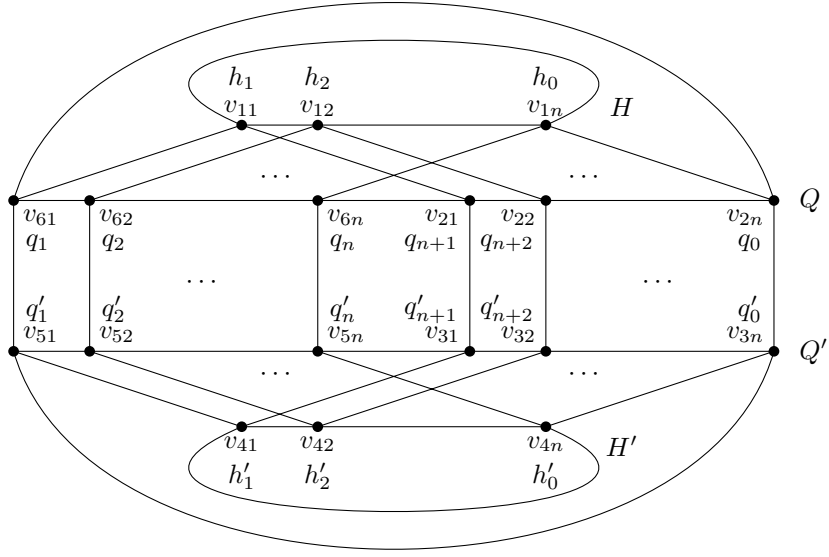
\begin{figure}[htbp]
\centering
\begin{tikzpicture}
\coordinate (v11) at (3,2);
\coordinate (v12) at (4,2);
\coordinate (v15) at (7,2);
\coordinate (v21) at (0,1);
\coordinate (v22) at (1,1);
\coordinate (v25) at (4,1);
\coordinate (v26) at (6,1);
\coordinate (v27) at (7,1);
\coordinate (v28) at (8,1);
\coordinate (v20) at (10,1);
\coordinate (v31) at (0,-1);
\coordinate (v32) at (1,-1);
\coordinate (v35) at (4,-1);
\coordinate (v36) at (6,-1);
\coordinate (v37) at (7,-1);
\coordinate (v30) at (10,-1);
\coordinate (v41) at (3,-2);
\coordinate (v42) at (4,-2);
\coordinate (v45) at (7,-2);
\coordinate (H) at (8,2);
\coordinate (H') at (8,-2);
\draw
(H)node[above]{$H$}
(H')node[below]{$H'$}
(v11)node[above]{$v_{11}$}
(v12)node[above]{$v_{12}$}
(v15)node[above]{$v_{1n}$}
(v41)node[below]{$v_{41}$}
(v42)node[below]{$v_{42}$}
(v45)node[below]{$v_{4n}$}
(v11)node[above=10]{$h_1$}
(v12)node[above=10]{$h_2$}
(v15)node[above=10]{$h_0$}
(v41)node[below=10]{$h_1'$}
(v42)node[below=10]{$h_2'$}
(v45)node[below=10]{$h_0'$}
(v21)node[below right]{$v_{61}$}
(v22)node[below right]{$v_{62}$}
(v25)node[below right]{$v_{6n}$}
(v26)node[below left]{$v_{21}$}
(v27)node[below left]{$v_{22}$}
(v20)node[below left]{$v_{2n}$}
(v31)node[above right]{$v_{51}$}
(v32)node[above right]{$v_{52}$}
(v35)node[above right]{$v_{5n}$}
(v36)node[above left]{$v_{31}$}
(v37)node[above left]{$v_{32}$}
(v30)node[above left]{$v_{3n}$}
(v21)node[below=16, right=1]{$q_1$}
(v22)node[below=16, right=1]{$q_2$}
(v25)node[below=16, right=1]{$q_n$}
(v26)node[below=16, left=1]{$q_{n+1}$}
(v27)node[below=16, left=1]{$q_{n+2}$}
(v20)node[below=16, left=1]{$q_0$}
(v31)node[above=16, right=1]{$q_1'$}
(v32)node[above=16, right=1]{$q_2'$}
(v35)node[above=16, right=1]{$q_n'$}
(v36)node[above=16, left=1]{$q_{n+1}'$}
(v37)node[above=16, left=1]{$q_{n+2}'$}
(v30)node[above=16, left=1]{$q_0'$}
(v20)node[right=6]{$Q$}
(v30)node[right=6]{$Q'$}
(v25)node[left=15, above=3]{$\cdots$}
(v27)node[right=15, above=3]{$\cdots$}
(v35)node[left=15, below=3]{$\cdots$}
(v37)node[right=15, below=3]{$\cdots$}
(2.5,1)node[below=25]{$\cdots$}
(8.5,1)node[below=25]{$\cdots$};
\draw
(v11)--(v15) (v21)--(v20)
(v11)--(v21) (v11)--(v26)
(v12)--(v22) (v12)--(v27)
(v15)--(v25) (v15)--(v20)
(v11) .. controls (-0.5,3.5) and (10.5, 3.5) .. (v15) 
(v21) .. controls (1,4.5) and (9,4.5) .. (v20);
\draw
(v41)--(v45) (v31)--(v30)
(v41)--(v31) (v41)--(v36)
(v42)--(v32) (v42)--(v37)
(v45)--(v35) (v45)--(v30)
(v41) .. controls (-0.5,-3.5) and (10.5,-3.5) .. (v45)
(v31) .. controls (1,-4.5) and (9,-4.5) .. (v30);
\draw
(v21)--(v31)
(v22)--(v32)
(v25)--(v35)
(v26)--(v36)
(v27)--(v37)
(v20)--(v30);
\foreach \point in 
{v21,v22,
v25,v26,v27,
v20,
v31,v32,
v35,v36,v37,
v30,
v11,v12,
v15,
v41,v42,
v45}
\fill[black,opacity=1] (\point) circle (2pt);
\end{tikzpicture}
\vskip -30pt
\caption{The graph $C_6\bowtie P_n$.}\label{fig:bowtie}
\end{figure}

Let us keep in mind that the subscript $i$ in the symbols~$h_i$ are considered modulo~$n$,
and the subscript~$i$ in the symbols~$q_i$ are modulo~$2n$, namely,
\[
h_{i+n}=h_i
\rmand
q_{j+2n}=q_j
\qquad\text{for all integers $i$ and $j$.}
\]

\begin{thm}\label{thm:bowtie}
Let $n\ge 5$ be an odd integer. Then the bow-tie graph $C_6\bowtie P_n$ is 3-extendable.
\end{thm}

\begin{proof}
Let $G$ be the bow-tie graph $C_6\bowtie P_n$.

Let $M_0$ be a $3$-matching of the graph~$G$.
We call an edge of~$M_0$ \emph{faithful} if it is an edge of the subgraph $G[J]$;
\emph{co-faithful} if it is an edge of the subgraph $G[J']$;
and \emph{unfaithful} otherwise, i.e., if it is an edge of the form~$q_jq_j'$ for some $j\in[2n]$.
Correspondingly, we call a vertex of the matching~$M_0$ faithful (resp., co-faithful, unfaithful) 
if it is a vertex of a faithful (resp., co-faithful, unfaithful) edge.

Suppose that the matching~$M_0$ has $x$ faithful edges, $y$ unfaithful edges,
and~$z$ co-faithful edges. Then we have $x+y+z=3$.
By the symmetry of the graph~$G$, we can suppose that $x\ge z$. 
Then we have $z=0$ or $z=1$.
For each of them, we will construct a perfect matching of the graph~$G$ which extends the matching~$M_0$.

The following lemma serves for \cref{lem2}, by which we can solve the case $z=0$.
The other case $z=1$ can be divided into the cases $(x,y,z)=(1,1,1)$
and $(x,y,z)=(2,0,1)$. We will handle them by \cref{lem3} and \cref{lem4} respectively.

\begin{lem}\label{lem1}
Every $3$-matching of the subgraph~$G[J]$ can be extended to a matching covering the vertex set~$H$.
\end{lem}

\begin{proof}
Suppose that the claim is false. 
Let~$M_0$ be a $3$-matching which is not extendable in this way.
Let~$\tilde{M}$ be an extension of the matching~$M_0$, 
which covers the maximum number of vertices in the set~$H$. 
Without loss of generality, we can suppose that $h_1\notin V(\tilde{M})$. 
By the choice of the matching~$\tilde{M}$, we infer that $q_1,q_{n+1}\in V(M_0)$.
Therefore, we can write 
\[
M_0=\{q_0q_1,\,e_2,\,e_3\},
\] 
where $e_2\in\{q_nq_{n+1},\,q_{n+1}q_{n+2}\}$.
We will find a contradiction by constructing an extension of the matching~$M_0$,
which covers the vertex set~$H$.
It suffices to find a matching~$M$ of the subgraph $G[H]-V(M_0)$ such that 
the subgraph $G[H]-V(e_3)-V(M)$ consists of paths of even orders.
We proceed according to the number of vertices in~$H$ covered by the edge~$e_3$. 

\medskip
\noindent{\it Case 1.}
If $V(e_3)\cap H=\emptyset$, then we can define
\[
M=\begin{cases}
\{h_2q_2\},&\text{if $q_2\notin V(e_3)$},\\[4pt]
\{h_{n-1}q_{n-1}\},&\text{otherwise}.
\end{cases}
\]

\noindent{\it Case 2.}
If $|V(e_3)\cap H|=1$, then we can define $M=\emptyset$.

\medskip
\noindent{\it Case 3.}
If $|V(e_3)\cap H|=2$, then we have $e_3=h_ih_{i+1}$, where $i\in\{2,3,\ldots,n-1\}$. We can define
\[
M=\begin{cases}
\{h_{n-1}q_{n-1}\},&\text{if $i$ is even and $i\ne n-1$};\\[4pt]
\{h_2q_2\},&\text{if $i$ is odd};\\[4pt]
\{h_3q_3\},&\text{if $i=n-1$}.
\end{cases}
\]
This proves \cref{lem1}.
\end{proof}

Here is the lemma by using which the case $z=0$ can be solved.

\begin{lem}\label{lem2}
Suppose that the matching $M_0$ has no co-faithful edges.
Then the subgraph $G[J]$ has a matching $M$ such that 
$M$ covers both the faithful edges and the vertex set~$H$, 
and that $M$ misses any unfaithful vertex. 
\end{lem}

\begin{proof}
We will prove it case by case, according to the number of faithful edges, say, $f$.

The case $f=3$ is \cref{lem1}.

When $f=2$, let $q_jq_j'$ be the unfaithful edge, where $j\in[2n]$. 
Assume that the vertex~$q_{j-1}$ is uncovered by the matching~$M_0$.
By \cref{lem1}, 
the $3$-matching $(M_0-q_jq_j')\cup \{q_{j-1}q_j\}$ 
can be extended to a matching~$M_1$,
which covers the set~$H$.
Then the matching $M_1-q_{j-1}q_j$ is a desired one.
For the same reason, 
\cref{lem2} holds true if the vertex~$q_{j+1}$ is uncovered by the matching~$M_0$. 
Now, we can suppose that both the vertices~$q_{j-1}$ and~$q_{j+1}$ are covered by~$M_0$.
By \cref{lem1}, 
the matching 
\[
M_2=(M_0-q_jq_j')\cup\{h_jq_{j+n}\}
\]
can be extended to a matching, say,~$M_2'$, which covers the set~$H$.
Since all the three neighbors~$q_{j-1}$, $q_{j+1}$ and~$h_j$,
of the vertex~$q_j$ in the subgraph~$G[J]$, are in the matching~$M_2$ which misses the vertex~$q_j$,
we infer that the unfaithful vertex~$q_j$ is not covered by the extended matching~$M_2'$.
Therefore, the matching~$M_2'$ is a desired one.

When $f=1$, we can represent the matching~$M_0$ as 
\[
M_0=\{e_1,\,q_jq_j',\,q_kq_k'\},
\]
where $1\le j<k\le 2n$.
If the vertices~$q_j$ and~$q_k$ are not adjacent in the subgraph $G[Q]$, 
namely, $|j-k|\ne1\pmod{2n}$,
then there exist two distinct vertices~$u$ and~$w$ such that 
\[
u\in \{q_{j-1},\,q_{j+1}\}\backslash V(e_1)
\rmand
w\in \{q_{k-1},\,q_{k+1}\}\backslash V(e_1).
\] 
By \cref{lem1},
the matching $\{uq_j,\,wq_k,\,e_1\}$ can be extended to
a matching, say,~$M_3'$, which covers the set~$H$. 
Then the matching $M_3'-uq_j-wq_k$ is a desired matching. 
Otherwise, the vertices~$q_j$ and~$q_k$ are adjacent.
By \cref{lem1},
the $2$-matching $\{q_jq_k\}\cup\{e_1\}$ 
can be extended to a matching, say,~$M_4'$, which covers the set~$H$. 
Then the matching $M_4'-q_jq_k$ is a desired matching. 

When $f=0$, the vertex set~$Q$ contains exactly three unfaithful vertices. 
Let~$q_j$ be a vertex in~$Q$ which is not unfaithful. 
Let~$M_5$ be the perfect matching of the path $H-h_j$ of order~$n-1$.
Then, the matching $M_5\cup\{q_jh_j\}$ is a desired matching.
This completes the proof of \cref{lem2}.
\end{proof}
\medskip

Now we deal with the first case $z=0$.
Let~$M$ be the matching obtained from \cref{lem2}.
Let~$M'$ be the matching of the subgraph~$G[J']$ 
which is symmetric to the matching~$M$.
In other words, an edge $h_i'h_{i+1}'$ (resp., $h_j'q_j'$, $q_j'q_{j+1}'$)
is in the matching~$M'$ if and only if the edge~$h_ih_{i+1}$ (resp., $h_jq_j$, $q_jq_{j+1}$) 
is in the matching~$M$.
Then the set 
\[
M\cup M'\cup \{q_jq_j'\,\colon\,q_j\in J-V(M)\}
\]
is a perfect matching of the graph~$G$ which covers the matching~$M_0$, as desired.

Next lemma is for the case $(x,y,z)=(1,1,1)$.

\begin{lem}\label{lem3}
For any edge $e$ in the subgraph~$G[J]$, and for any vertex $q_k$ in the set $Q-V(e)$,
the subgraph $G[J]-V(e)-q_k$ has a perfect matching. 
\end{lem}

\begin{proof}
It suffices to find a matching $M$ of the subgraph $G[J]-V(e)-q_k$, such that 
\begin{itemize}
\smallskip\item[(i)]
the path $G[H]-V(M)-V(e)$ is of even order;
\smallskip\item[(ii)]
every path component of the subgraph $G[Q]-V(M)-V(e)-q_k$ is of even order.
\end{itemize}

Below we will construct such a matching~$M$ according to the position of the edge~$e$.

\medskip\noindent{\it Case 1.}
If $V(e)\subset H$, 
we can suppose that $e=h_0h_1$ without loss of generality.
We can take the matching
\[
M=\begin{cases}
\{h_2q_2\},&\text{if $k$ is odd},\\[3pt]
\{h_2q_{n+2}\},&\text{if $k$ is even}.
\end{cases}
\]

\noindent{\it Case 2.}
If $V(e)\cap H\ne\emptyset$ and $V(e)\cap Q\neq\emptyset$,
then we can suppose that $e=h_1q_1$ without loss of generality.
We can take the matching
\[
M=\begin{cases}
\emptyset,&\text{if $k$ is even};\\[3pt]
\{h_0q_0,\,h_2q_2\},&\text{if $k$ is odd}.
\end{cases}
\]

\noindent{\it Case 3.} If $V(e)\subset Q$,
then we can suppose that $e=q_0q_1$ without loss of generality. 
We can take the matching
\[
M=\begin{cases}
\{h_2q_2\},&\text{if $k$ is odd};\\[3pt]
\{h_{n-1}q_{2n-1}\},&\text{if $k$ is even}.
\end{cases}
\]
This completes the proof.
\end{proof}

Now we are ready to solve the case $x=y=z=1$.
By \cref{lem3},
the subgraph $G[J]-V(M_0)$ has a perfect matching.
For the same reason,
the subgraph $G[J']-V(M_0)$ has a perfect matching.
The union of these two matchings and the matching~$M_0$ 
form a desired perfect matching of the graph~$G$.

For the last case $(x,y,z)=(2,0,1)$, we will need the following lemma.

\begin{lem}\label{lem4}
Let $e_0$ be an edge of the subgraph~$G[Q]$. 
Then any $2$-matching of the subgraph~$G[J]$ can be
extended to a near perfect matching of~$G[J]$, 
which covers the vertex set~$H\cup V(e_0)$. 
\end{lem}

\begin{proof}
Let $\{e_1,e_2\}$ be a $2$-matching of the subgraph~$G[J]$. 
It suffices to show that the subgraph~$G[J]-V(e_1\cup e_2)$ has a matching~$M$ such that
\begin{itemize}
\smallskip\item[(i)]
every path component of the subgraph $G[H]-V(e_1\cup e_2\cup M)$ is of even order;
\smallskip\item[(ii)]
at most one of the path components of 
the subgraph $G[Q]-V(e_1\cup e_2\cup M)$ is of odd order;
\smallskip\item[(iii)]
if the subgraph $G[Q]-V(e_1\cup e_2\cup M)$ has an isolated vertex,
then the isolated vertex is not an end of the edge $e_0$.
\end{itemize}
If such a matching $M$ exists, 
then the subgraph $G[Q]-V(e_1\cup e_2\cup M)$ has a near perfect matching~$M'$,
such that the matching $M\cup M'\cup\{e_1,e_2\}$ covers the vertex set $V(e_0)$.
The desired result follows immediately.

Below we will seek the above matching~$M$.
According to the positions of the edges~$e_1$ and~$e_2$,
we have $6$ cases to treat.

\medskip\noindent{\it Case 1.} 
Both the edges $e_1$ and~$e_2$ are from the subgraph~$G[H]$.

The subgraph $G[H]-V(e_1\cup e_2)$
consists of two paths of different parities of orders,
where the path of even order might be empty.
Let $h_i$ be an end of the path of odd order. 
We can take the matching $M=\{h_iq_i\}$.

\medskip
\noindent{\it Case 2.} 
The edge~$e_1$ is from the subgraph~$G[H]$,
and the edge $e_2=h_jq_j$ for some $j\in[2n]$.

We can suppose that $e_1=h_0h_1$ without loss of generality.
Then we can take the matching 
\[
M=\begin{cases}
\{h_jq_j\,\colon\ 2\le j\le n-1\},&\text{if $e_2\in\{h_jq_j\,\colon\ 2\le j\le n-1\}$};\\[4pt]
\{h_jq_j\,\colon\ n+2\le j\le 2n-1\},&\text{otherwise, i.e., if $e_2\in\{h_jq_j\,\colon\ n+2\le j\le 2n-1\}$}.
\end{cases}
\]

\noindent{\it Case 3.} 
The edge $e_1$ is from the subgraph $G[H]$,
and the edge $e_2$ is from the subgraph~$G[Q]$.

Without loss of generality, we can suppose that 
$e_2=q_0q_1$ and $e_1=h_ih_{i+1}$, where $0\le i\le n-1$.
Moreover, by symmetry, we can suppose that $0\le i\le (n-1)/2$ without loss of generality.
We can take the matching 
\[
M=\begin{cases}
\{h_{i+2}q_{i+2}\},&\text{if the vertex~$q_2$ is not an end of the edge~$e_0$};\\[4pt]
\{h_3q_{n+3}\},&\text{otherwise}.
\end{cases}
\]

\medskip
\noindent{\it Case 4.} 
Both the edges~$e_1$ and $e_2$ have the form $h_jq_j$, where $j\in[2n]$.

Without loss of generality, we can suppose that $e_1=h_1q_1$.
Then we can take the matching
\[
M=\begin{cases}
\{h_jq_j\,\colon\ 2\le j\le n\},&\text{if }e_2\in\{h_jq_j\,\colon\ 2\le j\le n\};\\[4pt]
\{h_jq_j\,\colon\ n+2\le j\le 2n\},&\text{otherwise, i.e., if }e_2\in\{h_jq_j\,\colon\ n+2\le j\le 2n\}.
\end{cases}
\]

\noindent{\it Case 5.} 
The edge $e_1$ has the form $h_jq_j$ for some $j\in[2n]$,
and the edge $e_2$ is from the subgraph~$G[Q]$.

Without loss of generality, we can suppose that $e_2=q_0q_1$, and that $j\in[n]$.
Then we can take the matching
\[
M=\begin{cases}
\{h_2q_2,\,h_4q_4\},&\text{if $q_2\in V(e_0)$, and $j=3$};\\[4pt]
\emptyset,&\text{otherwise}.
\end{cases}
\]

\noindent{\it Case 6.}
Both the edges $e_1$ and $e_2$ are from the subgraph $G[Q]$.
\medskip

The subgraph $G[Q]-V(e_1\cup e_2)$ consists of two paths, where one of them might be empty. 
Since the sum $2n-4$ of their orders is even, the two paths have the same parity of orders.
Since $2n-4\ge 6$, there is at most one path is of order~$1$.
If such an isolated vertex exists, say,~$q_j$, then we can take the matching~$M$ to be the edge~$h_jq_j$.
Otherwise, we can take the matching~$M$ to be the edge~$h_kq_k$,
where~$q_k$ is an end of the path whose order is larger, or an end of any path when the two paths have the same orders.

This completes the proof of \cref{lem4}.
\end{proof}

The last case $(x,y,z)=(2,0,1)$ can be done as follows.
Let $M_0=\{e_1,e_2,e_3\}$, where the edges~$e_1$ and~$e_2$ are faithful,
and the edge~$e_3$ is co-faithful.
By Lemma~\ref{lem4},
the subgraph~$G[J]$ has a near perfect matching~$M_1$ covering
the $2$-matching $\{e_1,e_2\}$, such that the associated uncovered vertex~$q_j$
satisfies that its symmetric vertex~$q_j'$ is uncovered by the edge~$e_3$.
By \cref{lem3}, the subgraph~$G[J']-V(e_3)-q_j$ has a perfect matching~$M_2$.
Hence the matching $M_1\cup M_2\cup \{e_3,\,q_jq_j'\}$ 
is a perfect matching of the graph~$G$ which extends the matching~$M_0$.
This completes the proof of \cref{thm:bowtie}.
\end{proof}

Since the bow-tie graph $C_6\bowtie P_n$ can be embedded onto the Klein bottle,
\cref{thm:bowtie} implies immediately that there is an infinite number of $3$-extendable graphs 
which are $N_2$-embeddable. For completeness, we pose the following conjecture.
\begin{conj}
For any even integer $m\ge6$ and any odd integer $n\ge 5$, 
the graph $C_m\bowtie P_n$ is $3$-extendable.
\end{conj}

\section{Proof of \cref{thm:k'}}\label{sec:pf}

Recall that $\mu'(\Sigma)$ is the minimum integer~$k$
such that there is no $\Sigma$-embeddable $k$-extendable non-bipartite graphs.
It follows that $\mu'(\Sigma)\le\mu(\Sigma)$.
This section is devoted to find out $\mu'(\Sigma)$.
We will need the following lemma.

\begin{lem}\label{lem:2k+2}
Let $k\ge 1$. Any connected $k$-extendable graph of order~$2k+2$ 
is either the complete graph~$K_{2k+2}$, or the complete bipartite graph~$K_{k+1,\,k+1}$.
\end{lem}
\begin{proof}
It is easy to show for the case $k=1$.
Below we let $k\ge 2$.
Let~$G$ be a connected $k$-extendable graph with $|G|=2k+2$.
By \cref{thm:LY:bip}, either the graph~$G$ is bipartite or we have $\kappa(G)\ge 2k$.

In the former case, any vertex in the part with larger order has degree at most the order of 
the other part, and thus, at most $|G|/2=k+1$.
By \cref{thm:ext:basic}, we have $\delta(G)\ge k+1$.
Therefore, both parts of the graph $G$ has order $k+1$.
Since $\delta(G)\ge k+1$, we infer that $G=K_{k+1,k+1}$.

In the latter case, we suppose to the contrary that the graph~$G$ is not complete.
Then~$G$ has a pair $(u,v)$ of non-adjacent vertices. 
Since the graph~$G$ is $k$-extendable and is of order~$2k+2$, 
we infer that the subgraph $G'=G-u-v$ does not have a perfect matching.
On the other hand, any graph~$H$ of even order with $\delta(H)\ge |H|/2$ has a Hamilton circuit; see Ore~\cite{Ore60}.
Since $\delta(G')\ge 2k-2\ge 2k/2=|G'|/2$, we deduce that the subgraph $G'$ has a Hamilton circuit,
and a perfect matching in particular, a contradiction. This completes the proof.
\end{proof}

Now we are in a position to show \cref{thm:k'}.

\begin{proof}
Let $\Sigma$ be a surface of characteristic~$\chi$.
Let $\mu'(\Sigma)$ to be the minimum integer $k$
such that every $\Sigma$-embeddable non-bipartite graph is not $k$-extendable.
Note that the inequality $\mu'(\Sigma)\le\mu(\Sigma)$ holds for any surface~$\Sigma$.

First, we deal with the sporadic cases that $\chi\ge -1$.
For the sphere~$S_0$,
we have $\mu'(S_0)\le\mu(S_0)=3$ by \cref{thm:k}.
On the other hand, it is clear that the graph $P_4\times C_5$ is planar and non-bipartite. 
Since it is $2$-extendable by \cref{thm:PC}, 
we deduce that 
\[
\mu'(S_0)=3.
\]
By \cref{thm:k}, we have $\mu(N_1)=3$. Thus, we infer that $\mu'(N_1)\le 3$.
Since every planar graph is $N_1$-embeddable, we deduce that $\mu'(N_1)\ge \mu'(S_0)=3$.
Therefore, we conclude that 
\[
\mu'(N_1)=3.
\]
For the torus $S_1$, we have $\mu'(S_1)\le\mu(S_1)=4$ by \cref{thm:k}.
On the other hand,
it is clear that the graph $C_6\times C_5$ is toroidal and non-bipartite.
Since it is $3$-extendable by \cref{thm:CC}, we infer that 
\[
\mu'(S_1)=4.
\]
For the Klein bottle $N_2$, we have $\mu'(N_2)\le\mu(N_2)=4$ by \cref{thm:k}.
On the other hand, 
the $N_2$-embeddable non-bipartite graph $C_6\bowtie P_5$ is $3$-extendable
by \cref{thm:bowtie}.
Thus we have 
\[
\mu'(N_2)=4.
\]
Along the same line, we have $\mu'(N_3)\le 4$ by \cref{thm:k}.
Since $\mu'(N_3)\ge \mu'(N_2)=4$, we infer that
\[
\mu'(N_3)=4.
\]

Below we can suppose that $\chi\le-2$.
Write
\begin{equation}\label{def:n}
n=\lfloor(7+\sqrt{49-24\chi})/4\rfloor.
\end{equation}
Since $\chi\le-2$, one may estimate that $n\ge4$.
By \cref{thm:g:K},
it is direct to check that the complete graph $K_{2n}$ is $\Sigma$-embeddable.
It is obvious that the graph~$K_{2n}$ is both $(n-1)$-extendable and non-bipartite.
Thus we have $\mu'(\Sigma)\geq n$.

Suppose, by way of contradiction, that $\mu'(\Sigma)>n$.
Then there exists a the graph~$G$, which is $\Sigma$-embeddable,
$n$-extendable, and non-bipartite.

The $n$-extendability implies that the graph $G$ has at least $2n+2$ vertices.
If $|G|=2n+2$, then the graph~$G$ must be the complete graph~$K_{2n+2}$ 
by \cref{lem:2k+2}.
By computing the genus and the non-orientable genus of the graph~$K_{2n+2}$ directly,
we see that $K_{2n+2}$ is not $\Sigma$-embeddable. Thus we have $|G|\ge 2n+4$.

Let $v$ be a control point in an embedding of the graph~$G$ on the surface~$\Sigma$.

Assume that $|G|\le 4n$.
By \cref{thm:LY:bip}, we deduce that the connectivity $\kappa(G)$ is at least~$2n$. 
It follows that
\[
d(v)\ge\delta(G)\ge\kappa(G)\ge2n.
\]
Let $x$ be the number of triangles containing~$v$.
Let $y=d(v)$. 

When $y=2n$, since the graph~$G$ is $n$-extendable, we infer that $x\le 2n-2$.
By \cref{lem:ctrl}, we deduce that
\begin{equation}\label{pf:2}
{2n+1\over6}
={y\over 4}-\frac{2n-2}{12}
\le{y\over 4}-{x\over12}
\le1-{\chi\over 2n+4}.
\end{equation}
Solving it we find that $n\le (1+\sqrt{81-24\chi})/4$.
By \cref{def:n}, we obtain that 
\[
{7+\sqrt{49-24\chi}\over 4}-1
<\biggl\lfloor{7+\sqrt{49-24\chi}\over 4}\biggr\rfloor
\le{1+\sqrt{81-24\chi}\over 4},
\]
which implies that $\chi>0$, a contradiction.
Otherwise $y\ge 2n+1$. By \cref{lem:ctrl}, we have
\[
{2n+1\over6}
\le{y\over6}
\le1-{\chi\over 2n+4},
\]
which is same to Ineq.~\eqref{pf:2} and thus impossible.

Now we are led to the case that $|G|\geq 4n+2$. 
Assume that $x\le 2n-2$.
By \cref{lem:ctrl} and \cref{lem:d:x}, 
we have
\[
\frac{n+1}{4}
\le{n+1+\lceil x/2\rceil\over4}-{x\over 12}
\le{y\over 4}-{x\over12}
\le1-{\chi\over 4n+2},
\]
Solving it we find that $n\le(5+\sqrt{49-16\chi})/4$.
By \cref{def:n}, we obtain that 
\begin{equation}\label{ineq:2}
{7+\sqrt{49-24\chi}\over 4}-1
<\biggl\lfloor{7+\sqrt{49-24\chi}\over 4}\biggr\rfloor
\le{5+\sqrt{49-16\chi}\over 4}.
\end{equation}
Solving the above inequality we find that $\chi\in\{-6,-5,-4,-3,-2\}$.
Substituting each of these five values of $\chi$ into Ineq.~\eqref{ineq:2},
we get a contradiction.
Otherwise, we have $x\ge2n-1$. Then \cref{lem:ctrl} gives that 
\[ 
{n+1\over 4}
\le {2n+1\over 6}={y\over 6}
\le1-{\chi\over 4n+2},
\]
the same contradiction.
This completes the proof of \cref{thm:k'}.
\end{proof}

At the end of this paper, we would like to share the happy approach of finding Formula~\eqref{ans:mu'}.
Our previous result~\cite{LW13X} on $(n,k)$-graphs is as follows.

\begin{thm}[Lu, Wang]\label{thm:nk}
Let $\Sigma$ be a surface of characteristic $\chi$.
Let $\mu(n,\Sigma)$ be the minimum integer~$k$ such that there is no $\Sigma$-embeddable $(n,k)$-graphs.
Then for $n\ge 1$, we have
\begin{equation}\label{ans:nk}
\mu(n,\Sigma)=\begin{cases}
\max(0,\,3-\lceil n/2\rceil),&\text{if the surface~$\Sigma$ is homeomorphic to the sphere};\\[3pt]
\max(0,\,\lfloor (\,7-2n+\sqrt{49-24\chi}\,)/4\rfloor),&\text{otherwise}.
\end{cases}
\end{equation}
\end{thm}

While Formula~(\ref{ans:nk}) was obtained by laborious computations,
we discover Formula~(\ref{ans:mu'}) by the guess-and-check strategy.
Although $(0,k)$-graphs are exactly $k$-extendable graphs, Formula~(\ref{ans:nk}) 
is valid under the premise $n\ge1$.
Nevertheless, it was Formula~(\ref{ans:nk}) by which we were inspired to guess Formula~(\ref{ans:mu'}) out.

\section*{Acknowledgements}
Lu is supported by the National Natural Science Foundation of China
(Grant Nos.~$11101329$ and~$11471257$).
Wang is supported by the National Natural Science Foundation of China (Grant No.~$11101010$).


\begin{thebibliography}{99}

\bibitem{AKP08}
R.E.L. Aldred, K. Kawarabayashi, M.D. Plummer,
On the matching extendability of graphs in surfaces,
{\sl J. Combin. Theory Ser. B}~{\bf 98} (2008), 105--115.

\bibitem{BW09B}
L.W. Beineke, R.J. Wilson, J.L. Gross, and T.W. Tucker, 
Topics in Topological Graph Theory,
Cambridge Univ.\ Press, 2009.

\bibitem{Dean92}
N. Dean,
￼The matching extendability of surfaces,
{\sl J. Combin. Theory Ser. B}~{\bf 54} (1992), 133--141.

\bibitem{Fra34}
P. Franklin,
A six colour problem,
{\sl J. Math. Phys.}~{\bf 13} (1934), 363--369.

\bibitem{GT87B}
J.L. Gross and T.W. Tucker,
Topological Graph Theory,
Wiley, 1987, and Dover, 2001.

\bibitem{GP92}
E. Gy\"ori, M.D. Plummer, 
The Cartesian product of a $k$-extendable and an $l$-extendable graph
is $(k+l+1)$-extendable, 
{\sl Discrete Math.}~{\bf 101} (1992), 87--96.

\bibitem{LL98}
J. Lakhal, L. Litzler,
A polynomial algorithm for the extendability problem in bipartite graphs,
￼{\sl Inform. Process. Lett.}~{\bf 65} (1998), 11-16.

\bibitem{Leb40}
H. Lebesgue,
Quelques cons\'equences simples de la formule d'Euler,
{\sl J. de Math.}~{\bf 9} (1940), 27--43.

\bibitem{LY93}
J. Liu and Q. Yu,
Matching extensions and products of graphs, 
{\sl Ann. Discrete Math.}~{\bf 55} (1993), 191--200.

\bibitem{LY04}
D. Lou and Q. Yu,
Connectivity of $k$-extendable graphs with large $k$,
{\sl Discrete Appl. Math.}~{\bf 136} (2004), 55-61.

\bibitem{LP86B}
L. Lov\'asz and M.D. Plummer, 
Matching Theory, Vol. 29, North-Holland Publishing Co., 1986.

\bibitem{LW13X}
H. Lu and D.G.L. Wang,
Surface embedding of $(n,k)$-extendable graphs,
to appear in Discrete Appl. Math.

\bibitem{Ore60}
O. Ore, Note on Hamilton circuits, 
{\sl Amer. Math. Monthly}~{\bf 67(1)} (1960), 55.

\bibitem{Ore67}
O. Ore,
The Four-Color Problem,
Academic Press, New York, 1967, pp. 54--61.

\bibitem{OP69}
O. Ore and M.D. Plummer,
Cyclic colorations of plane graphs, in: W.T. Tutte (Ed.), Recent Progress in Combinatorics,
Academic Press, New York, 1969, pp. 287--293.

\bibitem{PPPV87}
T. Parsons, G. Pica, T. Pisanski, A. Ventre,
Orientably simple graphs,
{\sl Math. Slovaca}~{\bf 37} (1987), 391--394.

\bibitem{Plu80}
M.D. Plummer,
On $n$-extendable graphs,
{\sl Discrete Math.}~{\bf 31} (1980), 201--210.

\bibitem{Plu88}
M.D. Plummer,
Matching extension and the genus of a graph,
{\sl J. Combin. Theory Ser. B}~{\bf 44} (1988), 329--337.

\bibitem{Plu88PB}
M.D. Plummer,
A theorem on matchings in the plane,
in C. T. B. T. Lars Dovling Andersen,
Ivan Tafteberg Jakobsen and P.D. Vestergaard (Eds.),
Graph Theory in Memory of G.A. Dirac,
vol. 41, Elsevier, 1988, 347--354.

\bibitem{Plu94}
M.D. Plummer,
Extending matchings in graphs: A survey,
{\sl Discrete Math.}~{\bf 127} (1994) 277--292.

\bibitem{Plu08BC}
M.D. Plummer,
Recent Progress in Matching Extension,
in: Building Bridges, 
{\sl Bolyai Soc. Math. Stud.}~{\bf 19}, 
M.~Gr\"otschel, G.O.H.~Katona, G.~S\'agi (Eds.),
Springer Berlin Heidelberg,
2008, 427--454.

\bibitem{Rin54}
G. Ringel,
Bestimmung der Maximalzahl der Nachbargebiete auf nichtorientierbaren Fl\"achen,
{\sl Math. Ann.}~{\bf 127} (1954), 181--214.

\bibitem{RY68}
G. Ringel and J.W.T. Youngs,
Solution of the Heawood map-coloring problem,
{\sl Proc. Natl. Acad. Sci. USA}~{\bf 60} (1968), 438--445.

\bibitem{Riz00}
R. Rizzi,
A short proof of K\"onig’s matching theorem,
{\sl J. Graph Theory}~{\bf 33}(3) (2000), 138--139.

\bibitem{Tut47}
W.T. Tutte, The factorization of linear graphs, 
{\sl J. London Math. Soc.}~{\bf 22(2)} (1947), 107--111.

\bibitem{You63}
J.W.T. Youngs,
Minimal imbeddings and the genus of a graph,
{\sl J. Math. Mech.}~{\bf 12} (1963), 303--315.

\bibitem{YL09B}
Q. Yu and G. Liu,
Graph Factors and Matching Extensions,
Higher Education Press,
Springer-Verlag, Berlin Heidelberg, 2009.


\end{thebibliography}
\end{document}